\documentclass[12pt,letterpaper]{amsart}

\usepackage{epsfig,amsmath,amsthm,amssymb,graphicx}
\usepackage[top=1.25in, bottom=1.25in, left=1.25in, right=1.25in]{geometry}
\newtheorem{proposition}{Proposition}[section]

\newtheorem*{thm_main}{Main Theorem}
\newtheorem*{thm_1}{Theorem 1}
\newtheorem*{cor_1}{Corollary 1}
\newtheorem*{cor_2}{Corollary 2}
\newtheorem*{cor_3}{Corollary 3}

\newtheorem{lemma}[proposition]{Lemma}
\theoremstyle{definition}
\newtheorem{definition}[proposition]{Definition}

\newcommand{\tb}{\text{tb}}
\newcommand{\rot}{\text{rot}}
\newcommand{\s}{\text{s}\,}
\newcommand{\g}{\text{g}}
\newcommand{\HFhat}{\widehat{HF}}
\begin{document}
\title[Satellite operators with distinct iterates]{Satellite operators with distinct iterates in smooth concordance}

\author{Arunima Ray$^{\dag}$}
\address{Department of Mathematics, Rice University MS-136\\
6100 Main St. P.O. Box 1892\\
Houston, TX 77251-1892}
\email{arunima.ray@rice.edu}
\urladdr{www.math.rice.edu/$\sim$ar25}

\thanks{$^{\dag}$Partially supported by NSF--DMS--1309081 and the Nettie S.\ Autrey Fellowship (Rice University)}
\date{\today}


\begin{abstract}
Let $P$ be a knot in an unknotted solid torus (i.e.\ a \textit{satellite operator} or \textit{pattern}), $K$ a knot in $S^3$ and $P(K)$ the satellite of $K$ with pattern $P$. For any satellite operator $P$, this correspondence gives a function $P:\mathcal{C} \rightarrow \mathcal{C}$ on the set of smooth concordance classes of knots. We give examples of winding number one satellite operators $P$ and a class of knots $K$, such that the iterated satellites $P^i(K)$ are distinct as smooth concordance classes, i.e.\ if $i \neq j \geq 0$, $P^i(K) \neq P^j(K)$, where each $P^i$ is unknotted when considered as a knot in $S^3$. This implies that the operators $P^i$ give distinct functions on $\mathcal{C}$, providing further evidence for the fractal nature of $\mathcal{C}$. There are several other applications of our result, as follows. By using topologically slice knots $K$, we obtain infinite families $\{P^i(K)\}$ of topologically slice knots that are distinct in smooth concordance. We can also obtain infinite families of 2--component links (with unknotted components and linking number one) which are not smoothly concordant to the positive Hopf link. For a large class of \textit{$L$--space knots} $K$ (including the positive torus knots), we obtain infinitely many prime knots $\{P^i(K)\}$ which have the same Alexander polynomial as $K$ but are not themselves $L$--space knots.
\end{abstract}

\maketitle

\section{Introduction}

A \textit{knot} is the image of a smooth embedding $S^1 \hookrightarrow S^3$. The satellite construction is a well-known function on $\mathcal{K}$, the set of isotopy classes of knots. Briefly, a satellite operator $P$ is a knot in a solid torus and the satellite knot $P(K)$ is obtained by tying the solid torus into the knot $K$. An example of the satellite construction is shown in Figure \ref{satellite}; a more precise definition is given in Section \ref{satelliteoperators}. 

\begin{figure}[h]
  \centering
  \includegraphics[width=5in]{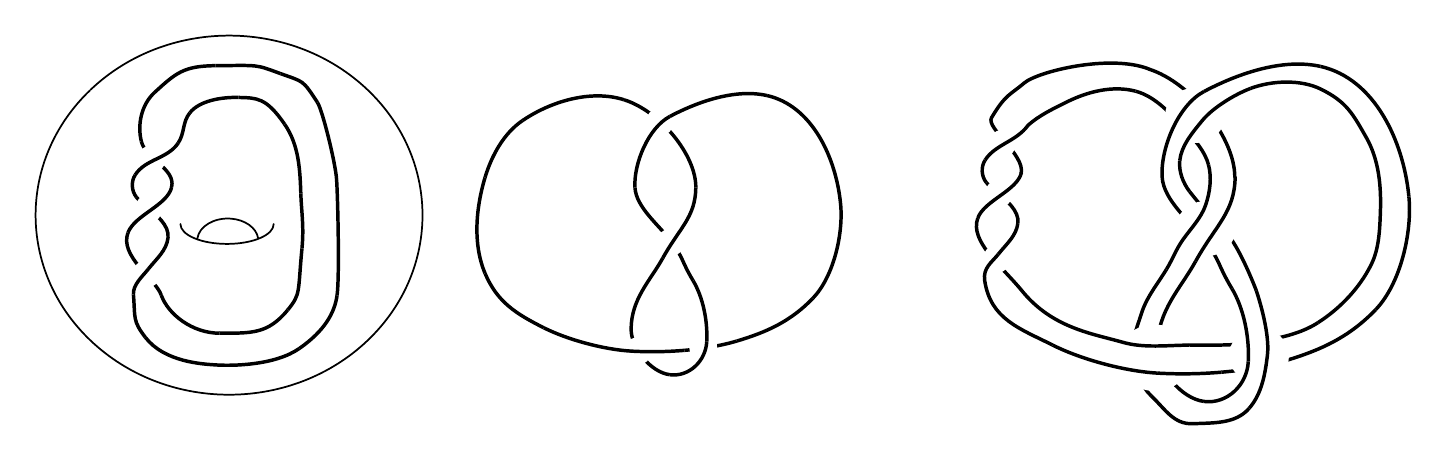}
  \put(-4.3,-0.20){$P$}
  \put(-2.75,-0.20){$K$}  
  \put(-1.05,-0.20){$P(K)$}  
  \caption{The satellite construction on knots.}\label{satellite}
\end{figure}

Two knots $K_0 \hookrightarrow S^3\times \{0\}$ and $K_1\hookrightarrow S^3\times\{1\}$ are said to be \textit{concordant} if they cobound a smooth, properly embedded annulus in $S^3\times [0,1]$. $\mathcal{K}$ modulo concordance forms an abelian group called the \textit{knot concordance group}, denoted by $\mathcal{C}$. Similarly, we say that two knots are \textit{exotically concordant} if they cobound a smooth, properly embedded annulus in a smooth 4--manifold \textit{homeomorphic} to $S^3\times [0,1]$ (but not necessarily \textit{diffeomorphic}). $\mathcal{K}$ modulo exotic concordance forms an abelian group called the \textit{exotic knot concordance group}, denoted by $\mathcal{C}^\text{ex}$. If the 4--dimensional smooth Poincar\'e Conjecture is true, we can see that $\mathcal{C}=\mathcal{C}^\text{ex}$ \cite{CDR14}. The satellite operation on knots descends to well-defined functions on $\mathcal{C}$ and $\mathcal{C}^\text{ex}$. 

Satellite knots are interesting both within and without knot theory. Satellite operations can be used to construct distinct knot concordance classes which are hard to distinguish using classical invariants, such as in \cite{CHL11,COT04}. In \cite{CFHeHo13}, winding number one satellite operators are used to construct non-concordant knots with homology cobordant zero--surgery manifolds. Satellite operations were used in \cite{H08} to subtly modify a 3--manifold without affecting its homology type. Winding number one satellite operators in particular are related to Mazur 4--manifolds \cite{AkKir79} and Akbulut corks \cite{Ak91}. 

There has been considerable interest in understanding the action of satellite operators on $\mathcal{C}$. For instance, it is a famous conjecture that the Whitehead double of a knot $K$ is smoothly slice if and only if $K$ is smoothly slice \cite[Problem 1.38]{kirbylist}. This question might be generalized to ask if operators are \textit{injective}, that is, given an operator $P$, does $P(K)=P(J)$ imply $K=J$ in smooth concordance? A survey of some recent work on the Whitehead doubling operator may be found in \cite{HeK12}. In \cite{CHL11}, several `robust doubling operators' were introduced and evidence was provided for their injectivity. Not much else is known in the winding number zero case. For operators with nonzero winding numbers, there has been more success. Recently Cochran, Davis and the author proved the following result.

\begin{thm_1}[Theorem 5.1 of \cite{CDR14}]If $P$ is a strong winding number one satellite operator, the induced function $P:\mathcal{C}^\text{ex}\rightarrow\mathcal{C}^\text{ex}$ is injective, i.e.\ for any two knots $K$ and $J$, $P(K)=P(J)$ if and only if $K=J$ in $\mathcal{C}^\text{ex}$. If the 4--dimensional smooth Poincar\'{e} Conjecture is true, $P:\mathcal{C}\rightarrow \mathcal{C}$ is injective.\end{thm_1}

The notion of a `strong winding number one' satellite operator is described in Section \ref{satelliteoperators}. In particular, any winding number one operator which is unknotted as a knot in $S^3$ is strong winding number one; see Figure \ref{P} for an example. 

\begin{figure}[t]
\centering
\includegraphics[width=2.5in]{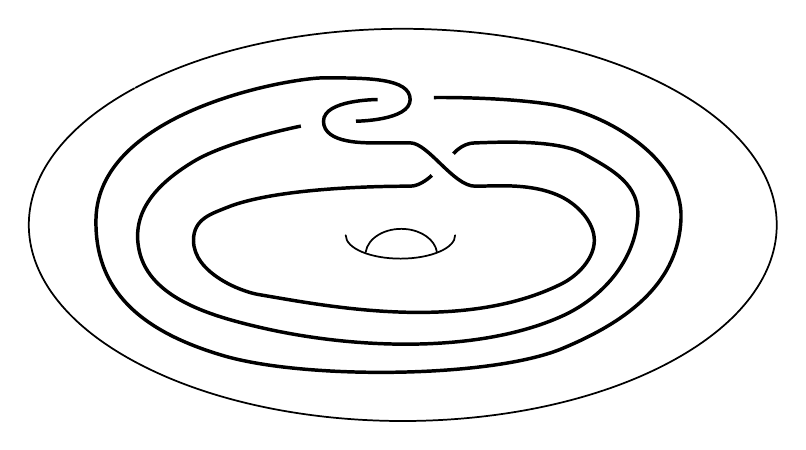}
\caption{The satellite operator $P$.}\label{P}
\end{figure}

Theorem 1 is related to the possibility of $\mathcal{C}$ having a \textit{fractal} structure. This was conjectured in \cite{CHL11} where some evidence was provided to support this theory. One may characterize `fractalness' of a set as the existence of \textit{self-similarities at arbitrarily small scales}. By Theorem 1, any strong winding number one satellite operator gives a self-similarity for $\mathcal{C}^\text{ex}$; however, while there exist several such operators (see \cite[Section 2]{CDR14}), the question of scale has not been addressed. This is the objective of the main theorem of this paper, which follows. 

\begin{thm_main}For any strong winding number one satellite operator $P$ with a Legendrian diagram where $\tb(P)>0$ and $\tb(P)+\rot(P)\geq 2$, (e.g.\ the one shown in Figure \ref{P}) and any knot $K$ with $\tb(K)=2\g(K)-1$, the knots $P^i(K)$ are distinct in $\mathcal{C}^\text{ex}$ and $\mathcal{C}$. That is, $P^i(K)\neq P^j(K)$ for all $i\neq j\geq 0$ in $\mathcal{C}$ and $\mathcal{C}^\text{ex}$. 

In particular, for the pattern $P$ in Figure \ref{P} and knots $K$ as above, $$\tau(P^i(K))=\g_4(P^i(K))=\g(P^i(K))=\g(K)+i=\g_4(K)+i=\tau(K)+i$$
\end{thm_main}

Examples of knots $K$ with the property that $\tb(K)=2\g(K)-1$ are plentiful. Any knot which is the closure of a positive braid (such as the positive torus knots) has this property. If a knot satisfies this condition, so does its untwisted Whitehead double \cite{Rud95}. In addition, this property is preserved under connected sum of knots. Such knots have the additional property that $\g(K)=\g_4(K)=\tau(K)$ (this is easily seen using the slice--Bennequin inequality---see Proposition \ref{slben}).

In addition to the operator $P$ shown in Figure \ref{P}, the main theorem applies to several other satellite operators. Two infinite families of such patterns are given in Figures \ref{leg_pattern_Qi} and \ref{leg_pattern_Ri}; see Proposition \ref{exactnumber_QR} for exact calculations of various invariants for these families. 

The action on $\mathcal{C}$ by the operators in the main theorem should be compared to shrinking the Cantor ternary set by a factor of three, namely that each iteration gives a distinct image of $\mathcal{C}$ at smaller and smaller scales. To complete the fractal analogy one must also address the question of \textit{surjectivity} of strong winding number one operators; some progress towards this end has been achieved by Davis and the author in \cite{DR13}. 

The question of whether the iterates of a satellite operator are distinct is an interesting question in its own right. In particular, there is no known counterpart of our main theorem for the Whitehead doubling operator. It was recently shown in \cite{KBPark14} that for the torus knots $T_{2,\,2m+1}$ with $m>2$, $\text{Wh}(T_{2,\,2m+1})$ and $\text{Wh}^2(T_{2,\,2m+1})$ are independent in $\mathcal{C}$; however, we are still unable to distinguish any of the other iterated Whitehead doubles of any knots.

\subsection{Applications of the main theorem}

Recall that given any knot $J$ with $\tb(J)=2\g(J)-1$, the untwisted Whitehead double of $K$ has the same property. Therefore, since any untwisted Whitehead double is topologically slice, we have several examples of topologically slice knots $K$ that we may use in our main theorem. For any of the operators $P$ as in the main theorem, each $P^i(K)$ will be topologically slice (since $P$ is unknotted as a knot in $S^3$) and $P^i(K)\neq P^j(K)$ for all $i\neq j\geq 0$ in smooth (as well as exotic) concordance. This gives us the following corollary.

\begin{cor_1}There exist infinite families of smooth (and exotic) concordance classes of topologically slice knots, where given any two knots in a family, one is a satellite of the other.\end{cor_1}

Several examples of infinite families of smooth concordance classes of topologically slice knots exist in the literature, such as in \cite{En95, Gom86, HeK12, Hom11}. Our examples are novel only due to the ease with which they are obtained and the added property that they are iterated satellites. 

We can also obtain an interesting corollary about \textit{$L$--space knots}, which are defined as follows. A homology sphere $Y$ is an $L$--space if $\HFhat(Y)$ is the same as that of a lens space (here $\HFhat$ is the Heegaard--Floer invariant introduced in \cite{OzSz04}). A knot $K$ is called an \textit{$L$--space knot} if some positive integer surgery on $S^3$ along $K$ yields an $L$--space. All positive torus knots, i.e.\ the knots $T_{p,\,q}$ with $p,q>0$, are well-known $L$--space knots, since $pq\pm1$ surgery on them yield lens spaces. $L$--space knots have received much interest lately since their knot Floer complexes may be computed directly from their Alexander polynomials \cite{OzSz05}. It was also shown in \cite{OzSz05} that there are strong restrictions on the Alexander polynomial of $L$--space knots. Since several $L$--space knots, we obtain the following corollary.

\begin{cor_2}For any $L$--space knot $K$ with $\tb(K)>0$, there exist infinitely many prime knots with the same Alexander polynomial as $K$ which are not themselves $L$--space knots.\end{cor_2}

As a third application of the main theorem, we can construct infinitely many links which are not smoothly concordant to the Hopf link. Any 2--component link $L=(P,\eta)$ with $\eta$ unknotted gives a satellite operator by considering the knot $P$ in the solid torus $S^3 - N(\eta)$, where $N(\eta)$ is a regular neighborhood of $\eta$. It is easy to see that if two such links are concordant they give identical functions on $\mathcal{C}^\text{ex}$ \cite[Proposition 2.3]{CDR14}. Given a satellite operator $P$, we may consider the associated 2--component link $(P,\eta(P))$ where $\eta(P)$ is the meridian of the solid torus containing $P$. 

\begin{cor_3}For any operator $P$ in the main theorem, the associated links $(P^i,\eta(P^i))$ yield distinct concordance classes of links with linking number one and unknotted components, which are each distinct from the class of the positive Hopf link. \end{cor_3}

In addition, we know from \cite[Corollary 2.2]{CFHeHo13} that if a winding number one satellite operator $P$ is unknotted (such as the one in Figure \ref{P}), then the zero--surgery manifolds on $P(K)$ and $K$ are homology cobordant, for any knot $K$. Recall that the $n$--solvable, positive, negative, and bipolar filtrations of $\mathcal{C}$ (from \cite{COT04} and \cite{CHHo13}, denoted by $\{\mathcal{F}_n\}$, $\{\mathcal{P}_n\}$, $\{\mathcal{N}_n\}$, and $\{\mathcal{B}_n\}$ respectively) can be defined in terms of the zero--surgery manifolds of knots. Therefore, if we start with a knot $K$ in $\mathcal{F}_n/\mathcal{F}_{n-1}$ (resp. $\mathcal{P}_n/\mathcal{P}_{n-1}$ or $\mathcal{B}_n/\mathcal{B}_{n-1}$) with $\tb(K)=2\g(K)-1$ and an unknotted operator $P$ for which the main theorem applies, we obtain a family $\{P^i(K)\}$ of infinitely many classes of knots also in $\mathcal{F}_n/\mathcal{F}_{n-1}$ (resp. $\mathcal{P}_n/\mathcal{P}_{n-1}$ or $\mathcal{B}_n/\mathcal{B}_{n-1}$). Since if $K$ has $\tb(K)=2\g(K)-1$ it has $\tau(K)>0$, we cannot directly use the same construction for $\{\mathcal{N}_n\}$, but the mirror images of the examples for $\{\mathcal{P}_n\}$ suffice. 


\subsection{Acknowledgements}
The author would like to thank Tim Cochran, Jung Hwan Park, and Christopher Davis for their time and insightful discussions. We are also indebted to the participants of the Heegaard Floer ``computationar'' held at Rice University in Spring 2014, particularly Allison Moore and Eamonn Tweedy, for their insights into $L$--space knots and Heegaard--Floer homology.

\section{Background}\label{background}

\subsection{Satellite operators}\label{satelliteoperators}
A \textit{satellite operator}, or \textit{pattern}, is a knot in the standard unknotted solid torus $V=S^1\times D^2$. The \textit{winding number} of a satellite operator $P$, denoted by $w(P)$, is the signed count of the number of intersections of $P$ with a generic meridional disk of $V$.

The set of satellite operators is a monoid in the following natural way. Given a satellite operator $P$ in a solid torus $V$, we see the following curves:
\begin{itemize}
\item $\mu(P)$, the meridian of $P$ within $V$,
\item $\lambda(P)$, the longitude of $P$ within $V$, 
\item $m(P)$, the meridian of $V$, and
\item $\ell(P)$, the longitude of $V$. 
\end{itemize}

Given operators $P$, $Q$ in solid tori $V(P)$, $V(Q)$, we construct the composed pattern $P\star\, Q$ as follows. Let $N(Q)$ be a regular neighborhood of $Q$ inside $V(Q)$. Glue $V(Q) - N(Q)$ and $V(P)$ by identifying $\mu(Q)\sim m(P)$ and $\lambda(Q)\sim \ell(P)$. The resulting 3--manifold is a solid torus. The image of $P$ inside this solid torus is the desired operator $P\star\, Q$. An example of this construction is shown in Figure \ref{monoidop}.

\begin{figure}[b]
  \centering
  \includegraphics[width=5in]{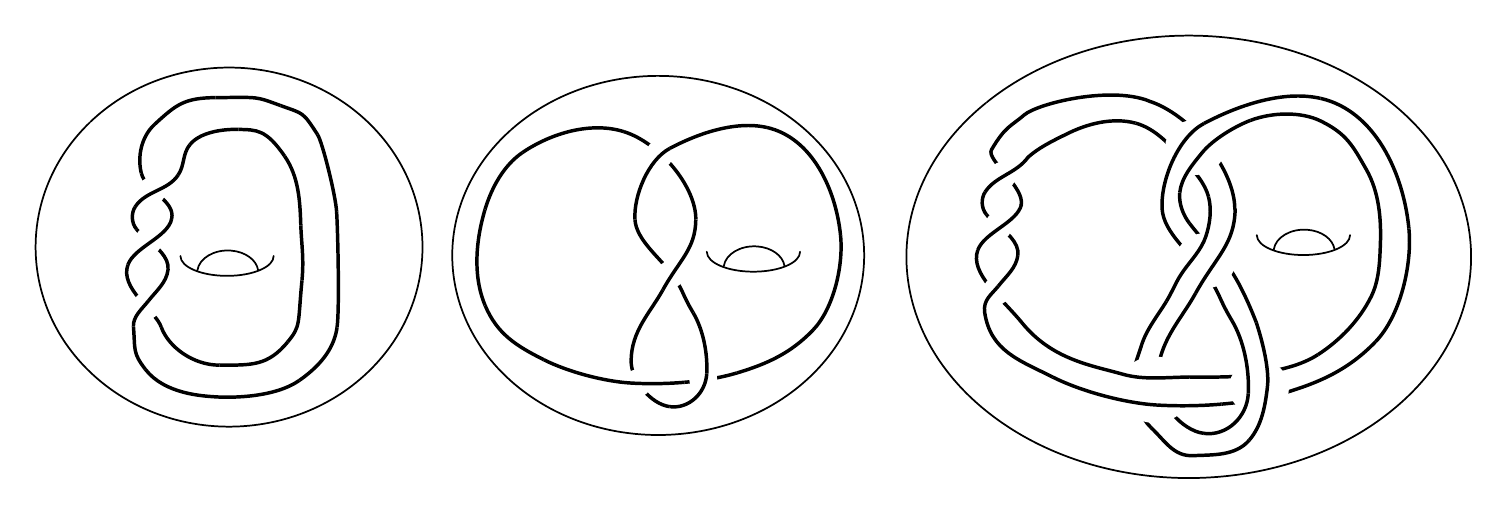}
  \put(-4.42,-0.20){$P$}
  \put(-2.9,-0.20){$Q$}  
  \put(-1.15,-0.20){$P\star Q$}  
  \caption{The monoid operation on satellite operators.}\label{monoidop}
\end{figure}

The well-known action of satellite operators on knots is closely related to the above construction. Given a knot $K$ and a satellite operator $P$ in a solid torus $V$, we obtain the satellite knot $P(K)$ as follows. Denote the meridian of $K$ by $\mu(K)$ and the longitude by $\lambda(K)$. Let $N(K)$ be a tubular neighborhood of $K$. Glue $S^3 - N(K)$ and $V$ by identifying $\mu(K)\sim m(P)$ and $\lambda(K)\sim\ell(P)$. The resulting 3--manifold is $S^3$ and the image of $P$ inside this manifold is the knot $P(K)$. An example of this construction is given in Figure \ref{satellite}. For a survey of the satellite construction see \cite[p. 10]{Lic97} or \cite[p. 111]{Ro90}.

It is easily seen that $(P\star\, Q)(K)=P(Q(K))$, i.e.\, the satellite construction gives a monoid action on $\mathcal{K}$. We denote the satellite operator $P\star P \star \cdots \star P$ by $P^i$. Therefore, $P^i(K)=P(P(\cdots(K)\cdots))=(P\star P \star \cdots \star P)K$, i.e.\ we get the same result whether we start with a knot $K$ and apply $P$ to it $i$ times or we apply the composed pattern $P\star P \star \cdots \star P$ to $K$ once.

Given a satellite operator $P$, we denote by $\widetilde{P}$ the knot $P(U)$, where $U$ is the unknot. $P$ is said to be \textit{strong winding number one} \cite[Definition 1.1]{CDR14} if $\mu(P)$ normally generates $\pi_1(S^3 - \widetilde{P})$. If $\widetilde{P}$ is unknotted, $H_1(S^3 - \widetilde{P})$ and $\pi_1(S^3 - \widetilde{P})$ are canonically isomorphic and thereofore $P$ is strong winding number one if and only if it winding number one \cite[Proposition 2.1]{CDR14}.

If the knots $K_0$ and $K_1$ are concordant in any 4--manifold $M$ bounded by two disjoint copies of $S^3$, the satellites $P(K_0)$ and $P(K_1)$ are concordant in $M$ for every operator $P\subseteq V$. This is easily seen as follows. Let $C$ be the concordance between $K_0$ and $K_1$. We excise a neighborhood of $C$ and glue in $V(P)\times [0,1]$. The image of $P\times [0,1]$ in the resulting manifold (which is diffeomorphic to $M$) is a concordance between $P(K_0)$ and $P(K_1)$. As a result, the satellite construction is well-defined on $\mathcal{C}$ and $\mathcal{C}^\text{ex}$. 

\subsection{Legendrian knots and the slice--Bennequin inequality}

\begin{figure}[b]
  \begin{picture}(3.4,2)
  \put(0,0.52){\includegraphics[width=1.5in]{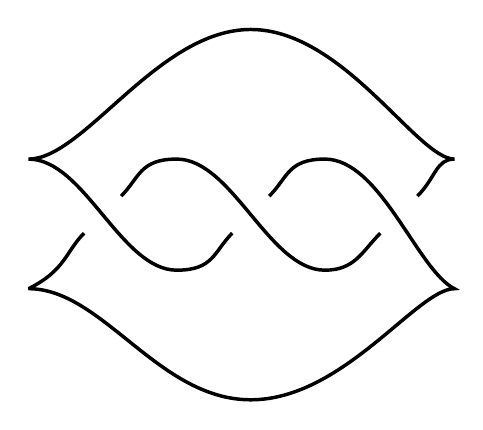}}
  \put(0.45,0){$\tb(K)=1$}
  \put(1.9,0.25){\includegraphics[width=1.5in]{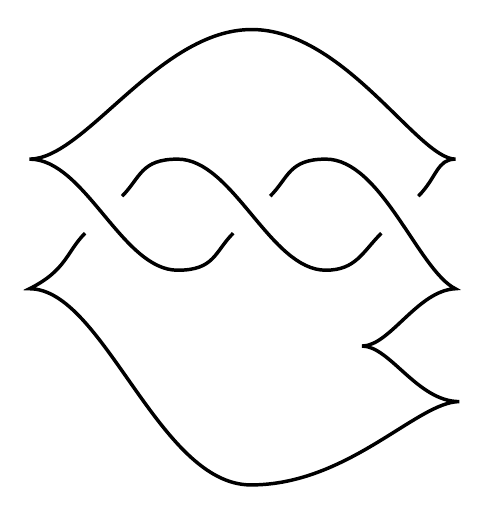}}
  \put(2.30,0){$\tb(K)=0$}
  \end{picture}
  \caption{Two different Legendrian realizations of the right-handed trefoil.}\label{leg_trefoils}
\end{figure}

An embedding of a knot $K$ in $S^3$ is said to be \textit{Legendrian} if at each point in $S^3$ it is tangent to the 2--planes of the standard contact structure on $S^3$. Legendrian knots can be studied concretely through their front projections as in Figure \ref{leg_trefoils}. Legendrian knots have two classical invariants, the \textit{Thurston--Bennequin number}, $\tb(\cdot)$, and the \textit{rotation number}, $\rot(\cdot)$, both of which can be easily calculated via front projections. See \cite{Etn05} for an excellent review of these and related concepts. 

Given a Legendrian knot with positive Thurston--Bennequin number, we may repeatedly stabilize at the cost of increasing the rotation number to get a Legendrian diagram with zero Thurston--Bennequin number; such a diagram is a different Legendrian knot, but has the same topological realization as the original knot. See Figure \ref{leg_trefoils} for an example.

We will make use of the slice--Bennequin inequality \cite{Rud95,Rud97}\cite[pp. 133]{Etn05}, which states that for any Legendrian knot $K$, 
$$\tb(K)+|\rot(K)|\leq 2\g_4(K)-1.$$
Here $\g_4(\cdot)$ stands for the smooth 4--genus of a knot, i.e.\ the least genus of a connected, oriented, smooth, properly embedded surface bounded by $K$ in $B^4$. Since some of our work will be in the exotic category, we show that the slice--Bennequin inequality has an exotic analog. 

\begin{definition}The \textit{exotic 4--genus} of a knot $K$, denoted by $g_4^\text{ex}(K)$, is the least genus of a connected, oriented, smooth, properly embedded surface bounded by $K$ in a manifold $\mathcal{B}$ where $\partial\mathcal{B}=S^3$ and $\mathcal{B}$ is homeomorphic (but not necessarily diffeomorphic) to $B^4$.\end{definition}

Exotic 4--genus is clearly an invariant of exotic concordance of knots, and is bounded above by the (classical) smooth 4--genus.

\begin{proposition}[Exotic slice--Bennequin equality]\label{slben}For a Legendrian knot $K$ in $S^3$ $$\tb(K)+|\rot(K)|\leq \s(K)-1 \leq 2g_4^\text{ex}(K)-1\leq 2g_4(K)-1$$ $$\tb(K)+|\rot(K)|\leq 2\tau(K)-1 \leq 2g_4^\text{ex}(K)-1\leq 2g_4(K)-1$$where $\s(K)$ is Rasmussen's invariant from Khovanov homology and $\tau(K)$ is Ozsv\'{a}th--Szab\'{o}'s invariant from Heegaard--Floer homology.\end{proposition}

\begin{proof}Corollary 1.1 of \cite{KronMrow13} shows that if $K$ bounds a connected, oriented, properly embedded surface $\Sigma$ in a homotopy 4--ball $\mathcal{B}$, then $\s(K)\leq 2\g(\Sigma)$. Similarly, from Theorem 1.1 of \cite{OzSz03}, in the special case of homotopy 4--balls, $\tau(K)\leq \g(\Sigma)$. From \cite{Plam04, Shuma07}, we know that 
$$\tb(K) + |\rot(K)|\leq \s(K)-1$$
and 
$$\tb(K) + |\rot(K)|\leq 2\tau(K)-1,$$
which completes the proof. \end{proof}

\subsection{The Legendrian satellite operation}

\begin{figure}[t]
\begin{picture}(4.4,2.65)
\put(0,0.75){\includegraphics{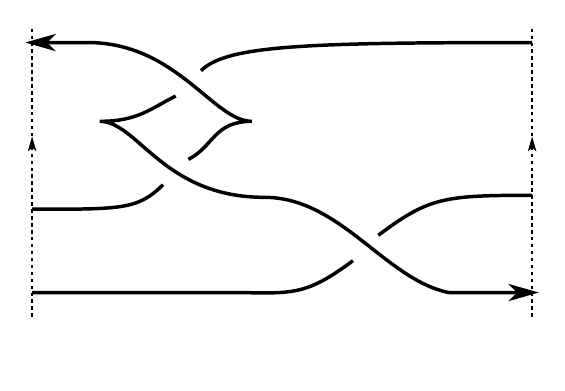}}
\put(0.62,0.6){$\tb(P)=2$}
\put(0.6,0.4){$\rot(P)=0$}
\put(0.9,0.1){(a)}
\put(2.4,0.45){\includegraphics{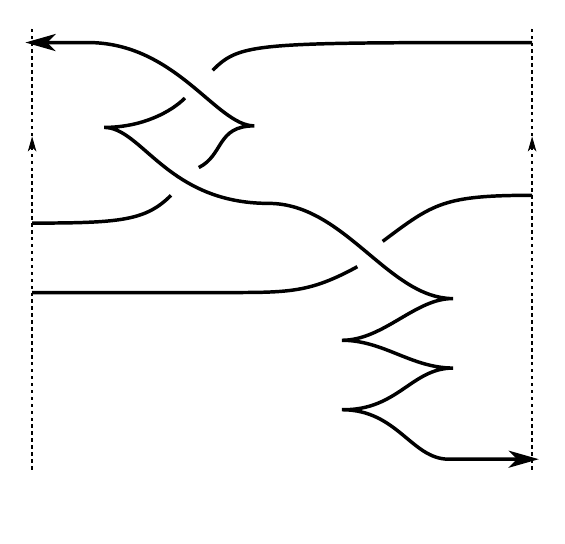}}
\put(3,0.6){$\tb(P)=0$}
\put(2.98,0.4){$\rot(P)=2$}
\put(3.3,0.1){(b)}
\end{picture}
\caption{Two Legendrian fronts for the pattern $P$ given in Figure \ref{P}. The dashed vertical lines are identified to yield a knot in $S^1\times \mathbb{R}^2$ endowed with its natural contact structure obtained as a quotient of $\mathbb{R}\times \mathbb{R}^2$.}\label{leg_patterns}
\end{figure}
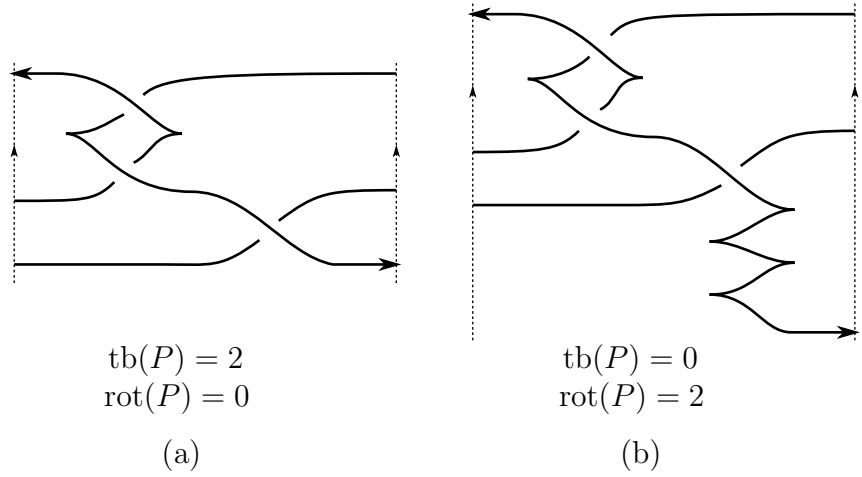

\begin{figure}[b]
\begin{picture}(5,3)
\put(0,0.75){\includegraphics[width=2.25in]{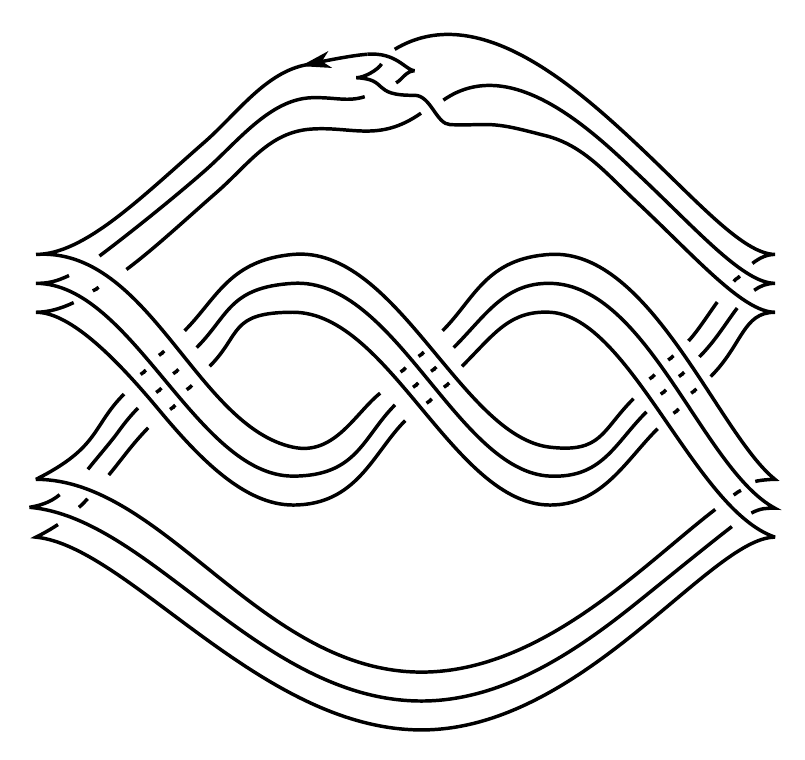}}
\put(1.12,0){(a)}
\put(2.7,0.25){\includegraphics[width=2.25in]{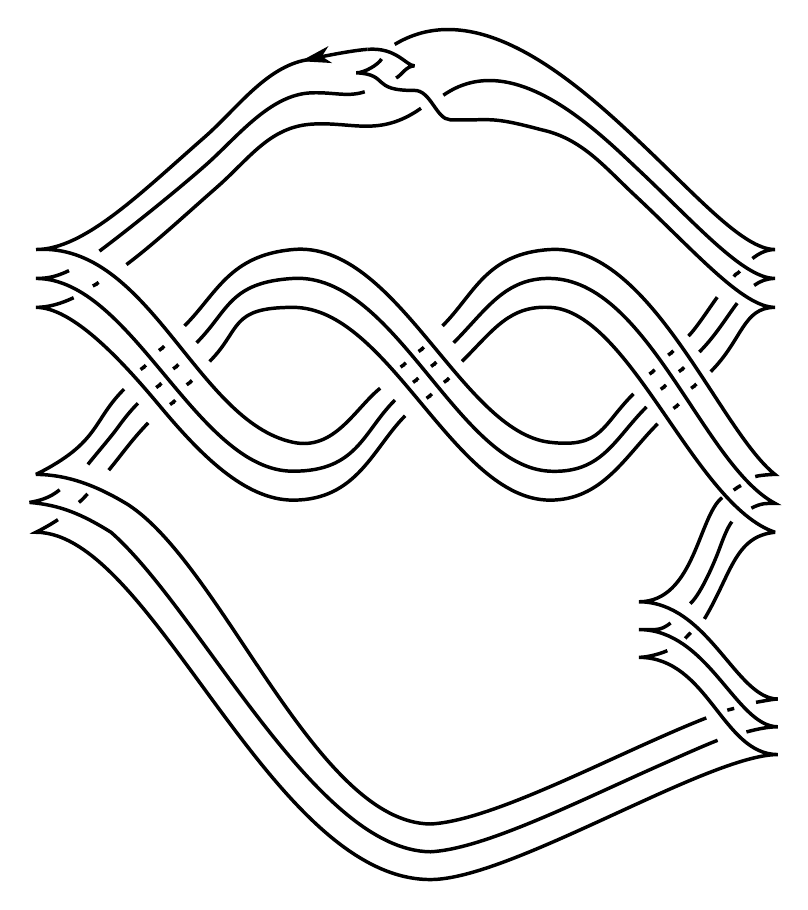}}
\put(3.82,0){(b)}
\end{picture}
\caption{The Legendrian satellite operation using $P$ in Figure \ref{leg_patterns}(a) with the Legendrian fronts shown in Figure \ref{leg_trefoils}. The knot in (a) is therefore the 1--twisted classical satellite of $K$ with pattern $P$ and the knot in (b) is the (classical) untwisted satellite of $K$ with pattern $P$.}\label{leg_satellite}
\end{figure}

The \textit{Legendrian satellite operation} is discussed in \cite{Ng01,NgTray04}. There, a Legendrian pattern $P$ in $S^1\times \mathbb{R}^2$ acts on a Legendrian knot $K$ in $\mathbb{R}^3$. A front diagram for a Legendrian pattern in shown in Figure \ref{leg_patterns}. In order to construct the (Legendrian) satellite knot $P(K)$ we take an $n$--copy of $K$ ($n$ `vertical' parallels of $K$) and insert $P$ in a strand of $K$, where $n$ is the number of strands of $P$. This process is described in Figure \ref{leg_satellite}. The resulting knot is the $\tb(K)$--twisted satellite of $K$ and $P$. If $\tb(K)=0$, the resulting knot is a Legendrian realization of the classical untwisted satellite knot $P(K)$. 

The same construction applies when a Legendrian pattern $P$ in $S^1\times \mathbb{R}^2$ acts on another pattern $Q$ also in $S^1\times \mathbb{R}^2$. This construction is described in Figure \ref{leg_Psquared}. The resulting operator $P\cdot Q$ is the $\tb(Q)$--twisted $P$--satellite of $Q$. Therefore, if $\tb(Q)=0$, $P\cdot Q$ corresponds to the operator $P\star Q$ described in Section \ref{satelliteoperators}. 

The following lemmata describe how the Thurston--Bennequin number and rotation number of satellites are related.

\begin{lemma}[Remark 2.4 of \cite{Ng01}]For a pattern $P$ and a knot $K$, 
\begin{equation}\label{tbformula}
\tb(P(K))=w(P)^2\tb(K)+\tb(P)
\end{equation}
\begin{equation}\label{rotformula}
\rot(P(K))=w(P)\rot(K)+\rot(P)\\
\end{equation}
\end{lemma}

\begin{lemma} For patterns $P$ and $Q$, 
\begin{equation}\label{tbpatternformula}
\tb(P\cdot Q)=w(P)^2\tb(Q)+\tb(P)
\end{equation}
\begin{equation}\label{rotpatternformula}
\rot(P\cdot Q)=w(P)\rot(Q)+\rot(P)\\
\end{equation}
\end{lemma}
\begin{proof} As in Remark 2.4 of \cite{Ng01}, these relationships are easily checked using front diagrams of $P$ and $Q$.\end{proof}


\begin{figure}[t]
\includegraphics[width=2.5in]{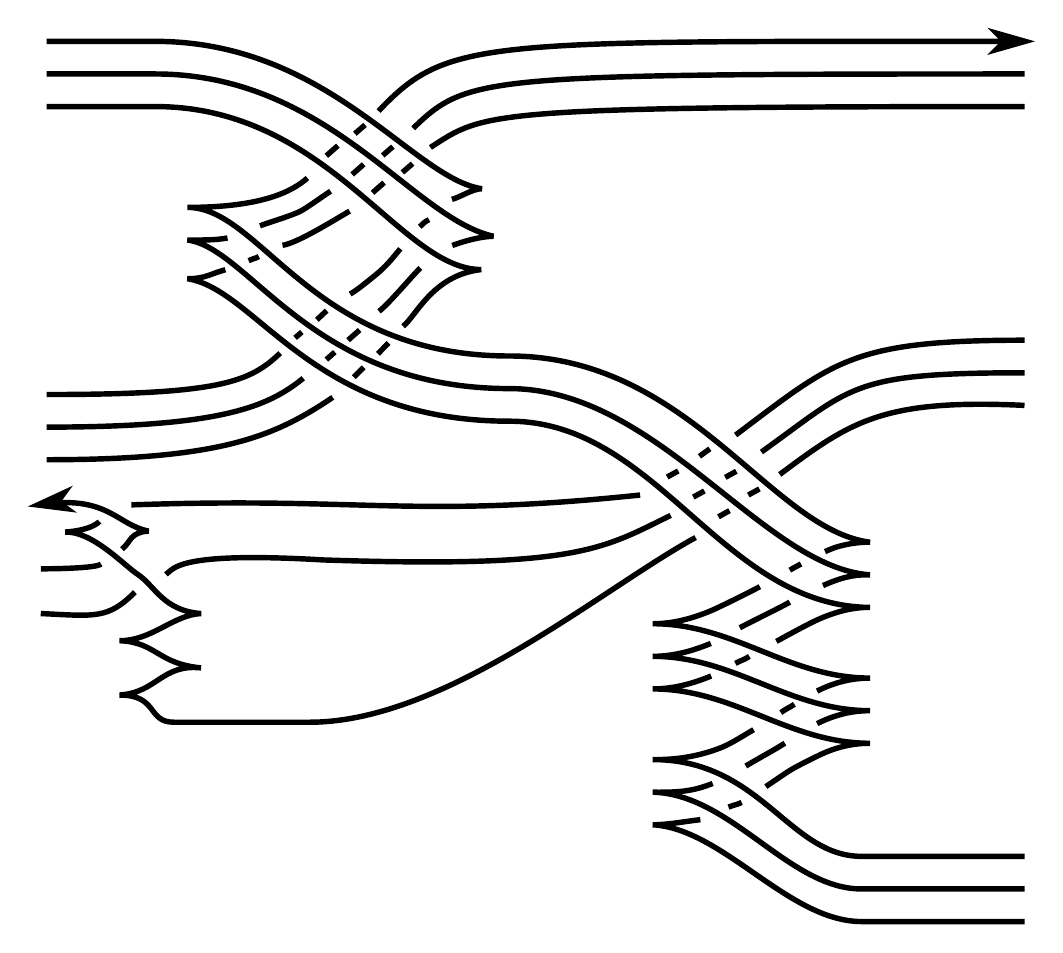}
\caption{The pattern $P^2$,                                       where $P$ is from Figure \ref{leg_patterns}(b).}\label{leg_Psquared}
\end{figure}
\section{Proof of the main theorem}\label{proof}

For the rest of this section, let $P$ denote the pattern shown in Figure \ref{leg_patterns}(b). We easily calculate that $\tb(P)=0$, $\rot(P)=2$ and $w(P)=1$. 

\begin{lemma}\label{tbrotpi}$\tb(P^i)=0$ and $\rot(P^i)=2i$.\end{lemma}
\begin{proof} Since $\tb(P)=0$, the Legendrian satellite construction coincides with the usual satellite construction. The rest is an easy consequence of Formulae (\ref{tbpatternformula}) and (\ref{rotpatternformula}).\end{proof}

For the rest of this section, fix a non-slice knot $K$ with a Legendrian diagram realizing $\tb(K)=0$ and $\rot(K)=2\g(K)-1$. There are many examples of such knots, as we mentioned in the introduction. It is easy to see using the (exotic) slice--Bennequin inequality, that such knots have the additional property that $\g(K)=\g_4(K)=\g_4^\text{ex}(K)=\tau(K)=\s(K)$.

\begin{proposition}\label{noteq1}$P^i(K)\neq K$ for any $i>0$, even in exotic concordance, where $P$ is the operator shown in Figure \ref{P}.\end{proposition}
\begin{proof}We can change $P^i(K)$ to $P^{i-1}(K)$ by changing a single positive crossing to a negative crossing. Therefore, by Corollary 1.5 of \cite{OzSz03}, we know that 
$$\tau(P^{i-1}(K))\leq \tau(P^i(K) \leq \tau(P^{i-1}(K))+1$$
and therefore, 
$$\tau(K)\leq \tau(P(K))\leq \tau(P^2(K)) \leq \cdots \leq \tau(P^i(K))$$
Recall that $\tau(\cdot)$ is an invariant of exotic concordance. Therefore, if  $P^i(K)=K$ even in exotic concordance for some $i>0$, then $\tau(P^i(K))=\tau(K)$. But this implies that for all $j$ with $0\leq j \leq i$, $\tau(P^j(K))=\tau(K)$. This contradicts Corollary 3.2 of \cite{CFHeHo13}, which shows exactly that $\tau(P(K))>\tau(K)$.\end{proof}

The following alternate proof uses the technique of the proof of Theorem 3.1 in \cite{CFHeHo13}, which shows that $P(K)\neq K$. 
\begin{proof}[Alternate proof of Proposition \ref{noteq1}]Using Formulae (\ref{tbformula}) and (\ref{rotformula}) and Lemma \ref{tbrotpi} we see that $$\tb(P^i(K))=0$$ $$\rot(P^i(K))=2\g(K)-1+2i$$ since $w(P^i)=1$. By the exotic slice--Bennequin equality, we have that $$0+\lvert 2\g(K)-1+2i \rvert\leq 2\g_4^\text{ex}(P^i(K))-1$$ Note that $g_4^\text{ex}(K)\leq \g(K)$ and $\g(K)\geq 1$. Therefore, $$g_4^\text{ex}(K) + i \leq g_4^\text{ex}(P^i(K))$$ Therefore, for $i>0$, $K \neq P^i(K)$ even in exotic concordance. 

Note that the above process also shows that 
$$\tau(K)+i\leq \tau(P^i(K))$$
and 
$$\s(K)+i\leq \s(P^i(K))$$
using the expanded versions of the exotic slice Bennequin inequality given in Proposition \ref{slben} and since $s(K)\leq 2g_4^\text{ex}(K)$ and $\tau(K)\leq g_4^\text{ex}(K)$ for any knot $K$.\end{proof}

\begin{proposition}\label{noteq2}Given $P$ and $K$ as above, $P^i(K) \neq P^j(K)$ for any $i\neq j$, even in exotic concordance.\end{proposition}
\begin{proof}We know from Theorem 5.1 of \cite{CDR14} that $P$ is an injective operator, i.e.\ if $P(J)=P(K)$ in exotic concordance for any two knots $J$ and $K$, we can infer that $J=K$ in exotic concordance. Therefore, if $P^i(K)=P^j(K)$ for some $i>j$, we would have that $P^{i-j}(K)=K$ in exotic concordance, which contradicts Proposition \ref{noteq1}.\end{proof}

We can actually make some stronger statements about the operators $P^i$.

\begin{proposition}\label{exactnumber}Given $P$ and $K$ as above, $$\tau(P^i(K))=\tau(K) + i$$ $$\g_4^{\text{ex}}(P^i(K))=\g_4(P^i(K))=\g_4(K)+i$$ $$\g(P^i(K))=\g(K)+i$$ for all $i\geq 0$. \end{proposition} 
Recall that $\g(K)=\g_4(K)=\tau(K)=\s(K)$ for the set of knots we are considering. Therefore, this proposition states that 
\begin{align*}
\tau(P^i(K))&=\g_4^\text{ex}(P^i(K))=\g_4(P^i(K))=\g(P^i(K))\\
&=\g(K)+i=\g_4(K)+i=\g_4^\text{ex}(K)+i=\tau(K)+i
\end{align*}
\begin{proof}[Proof of Proposition \ref{exactnumber}]The first statement is a consequence of Corollary 1.5 of \cite{OzSz03} as follows. We saw in the alternate proof of Proposition \ref{noteq1} that $\tau(K)+i\leq \tau(P^i(K))$. Since we can change $P^i(K)$ to $P^{i-1}(K)$ by changing a single positive crossing, by Corollary 1.5 of \cite{OzSz03}, we have that
$$\tau(P^i(K))\leq\tau(P^{i-1}(K))+1\leq\cdots\leq \tau(K)+i$$
Therefore, $\tau(P^i(K))=\tau(K)+i$. This gives an alternate proof that $P^i(K)\neq P^j(K)$ for $i\neq j$. 

In the alternate proof of Proposition \ref{noteq1}, we also saw that $g_4(K)+i \leq g_4^\text{ex}(P^i(K)) \leq g_4(P^i(K))$. Since $P^i(K)$ and $K$ are related by a sequence of $i$ crossing changes each of which can be accomplished by adding two bands as shown in Figure \ref{bandcrossingchange} we have that, $g_4^\text{ex}(P^i(K))\leq g_4(P^i(K))\leq g_4(K)+i$.

Since $\g_4(J)\leq\g(J)$ for any knot $J$, we must have that $\g(P^i(K))\geq\g_4(P^i(K))=\g_4(K)+i$. Since $\g_4(K)=\g(K)$, we have that $\g(P^i(K))\geq \g(K)+i$. One can construct a Seifert surface for $P^i(K)$ of genus $\g(K)+i$, as pointed out in \cite[Section 3]{CFHeHo13}. For $P(K)$, we can clearly see within the solid torus a genus one surface with two boundary components, one of which is the pattern $P$ and the other is the longitude of the solid torus. We glue this surface to a minimal genus Seifert surface for $K$, to see a Seifert surface for $P(K)$ with genus $\g(K)+1$. Since $\g_4(P(K))=\g(P(K))$, we can proceed by induction.
\end{proof}

\begin{figure}[t]
\centering
\includegraphics[width=4in]{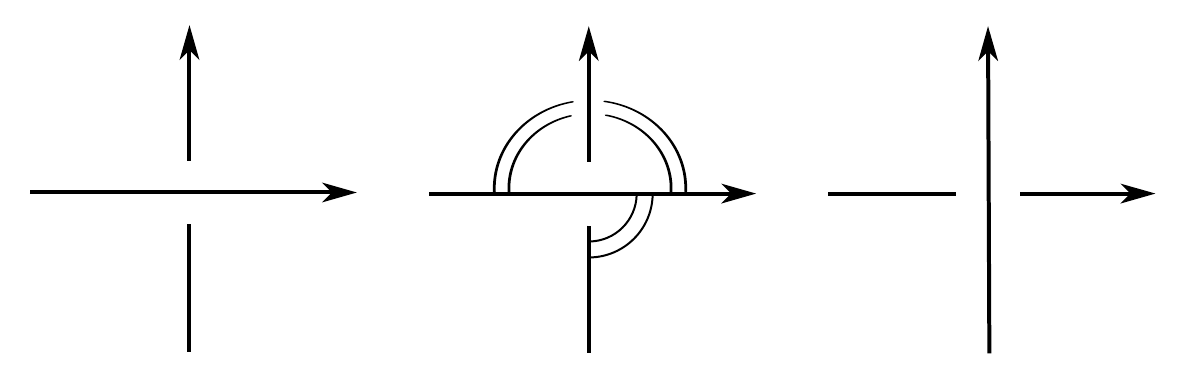}
\caption{Two band additions can effect a crossing change.}\label{bandcrossingchange}
\end{figure}

%

The techniques in the proof of the main theorem can be easily extended to several other patterns. In particular, consider the families of patterns $Q_j$ and $R_j$ shown in Figure \ref{leg_pattern_Qi} and Figure \ref{leg_pattern_Ri} respectively. Note that the pattern $Q_j$ is changed to $Q_{j-1}$ by changing a single positive crossing at the clasp and $Q_1$ is the pattern $P$ from Figure \ref{P}. Similarly, $R_j$ can be changed to $R_{j-1}$ by changing a single positive crossing at the top clasp, and $R_0$ is the identity operator (represented by the core of a solid torus).

\begin{figure}[t]
\centering
\includegraphics[width=5in]{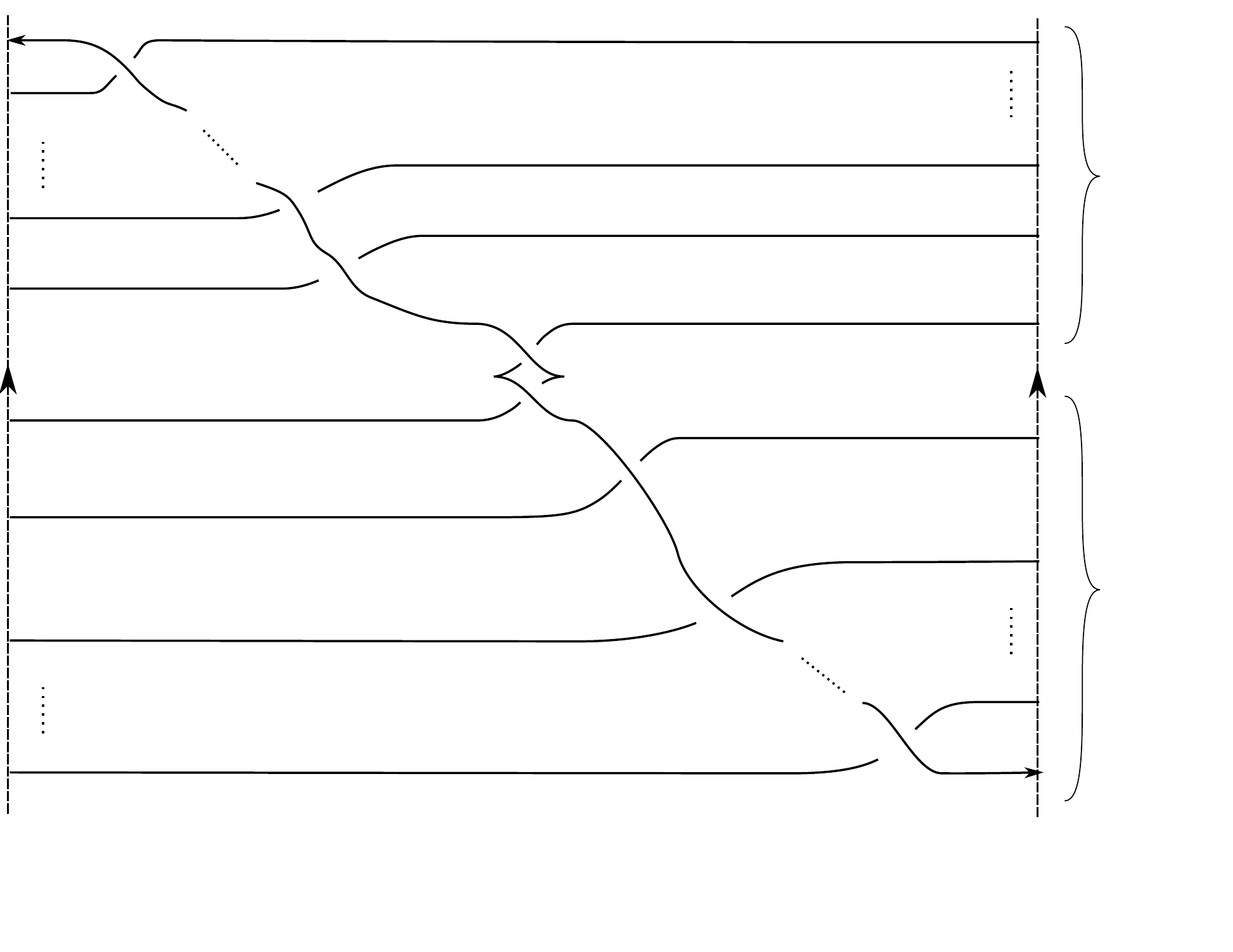}
\put(-0.45,3.15){$j$ strands}
\put(-0.45,1.45){$j+1$ strands}
\put(-2.87,0.4){$\tb(Q_j)=2j$}
\put(-2.87,0.2){$\rot(Q_j)=0$}
\put(-2.83,0){$w(Q_j)=1$}
\caption{A Legendrian diagram for the winding number 1 pattern $Q_j$. Notice that $Q_1$ is the pattern $P$ from Figure \ref{P}.}\label{leg_pattern_Qi}
\end{figure}

\begin{figure}[t]
\centering
\includegraphics[width=2.75in]{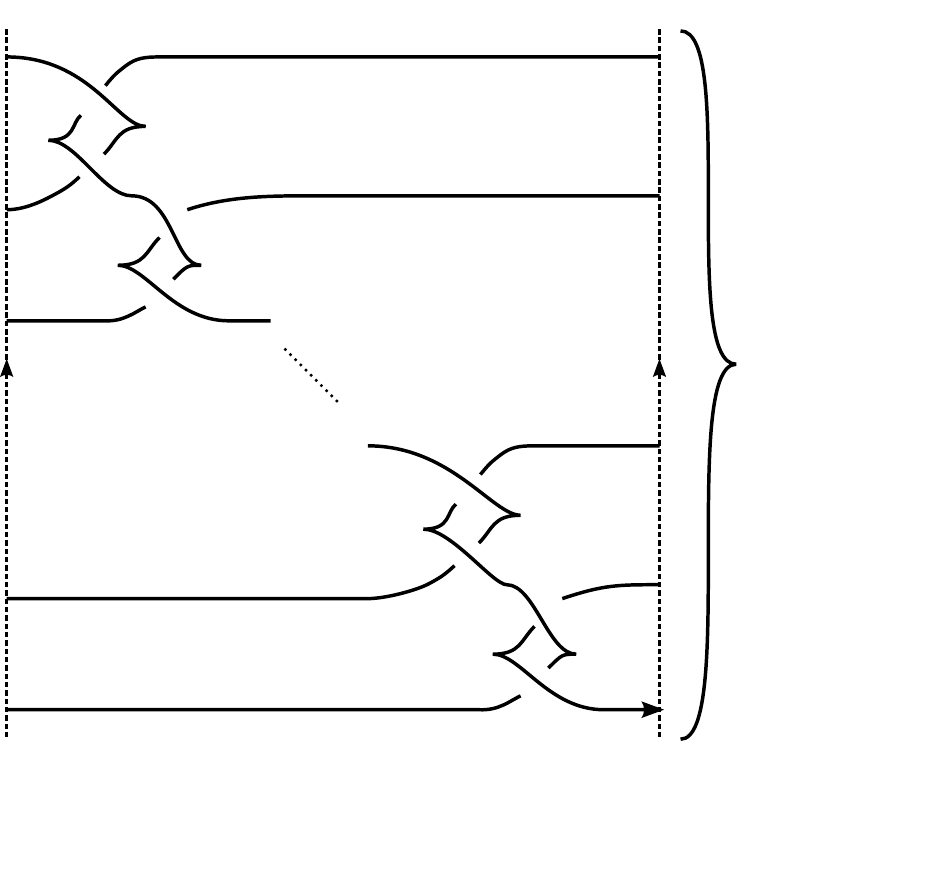}
\put(-0.55,1.5){$2j+1$ strands}
\put(-2.1,0.25){$\tb(R_j)=2j$}
\put(-2.1,0.05){$\rot(R_j)=0$}
\put(-2.05,-0.15){$w(R_j)=1$}
\caption{A Legendrian diagram for the winding number 1 pattern $R_j$. $R_j$ has $2j$ clasps and $R_0$ is identity satellite operator.}\label{leg_pattern_Ri}
\end{figure}
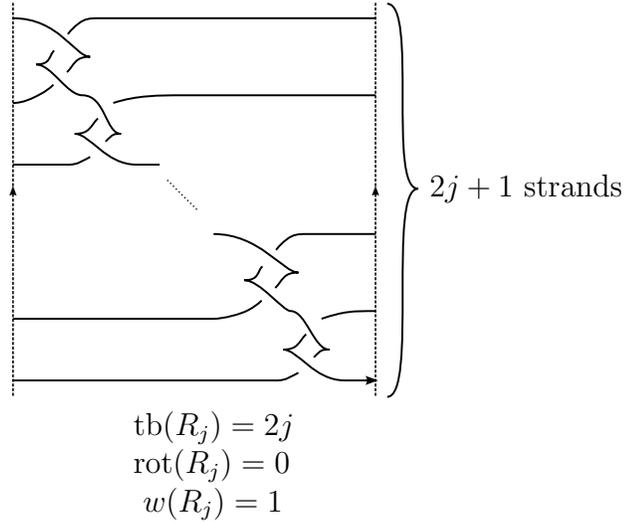

\begin{proposition}\label{exactnumber_QR}For the patterns $Q_j$ and $R_j$ and any $j\geq 0$ (shown in Figures \ref{leg_pattern_Qi} and \ref{leg_pattern_Ri}) and non-slice knots $K$ with Legendrian diagrams realizing $\tb(K)=0$ and $\rot(K)=2\g(K)-1$, we have that 
\begin{align*}
\tau(Q_j(K))&=\g_4^\text{ex}(Q_j(K))=\g_4(Q_j(K))=\g(Q_j(K))\\
&=\g(K)+j=\g_4(K)+j=\g_4^\text{ex}(K)+j=\tau(K)+j
\end{align*}
\begin{align*}
\tau(R_j(K))&=\g_4^\text{ex}(R_j(K))=\g_4(R_j(K))=\g(R_j(K))\\
&=\g(K)+j=\g_4(K)+j=\g_4^\text{ex}(K)+j=\tau(K)+j
\end{align*}

For the iterated satellite operators for $Q_j$ and $R_j$, we obtain the following.
\begin{align*}
\tau(Q_j^i(K))&=\g_4^\text{ex}(Q_j^i(K))=\g_4(Q_j^i(K))=\g(Q_j^i(K))\\
&=\g(K)+ij=\g_4(K)+ij=\g_4^\text{ex}(K)+ij=\tau(K)+ij
\end{align*}
\begin{align*}
\tau(R_j^i(K))&=\g_4^\text{ex}(R_j^i(K))=\g_4(R_j^i(K))=\g(R_j^i(K))\\
&=\g(K)+ij=\g_4(K)+ij=\g_4^\text{ex}(K)+ij=\tau(K)+ij
\end{align*}
\end{proposition}

We omit the proof of the above proposition since it is virtually identical to the proof of the main theorem. 
We see that the above statements yield the conditions we obtained in the proof of Proposition \ref{exactnumber} for $Q_1=P$, when we set $j=1$.

We also see easily that our proof works for any satellite operators $S$, which are strong winding number one and have Legendrian realizations realizing $\tb(P)>0$ and $\tb(P)+\rot(P)\geq 2$. In short, the proof would consist of using the Alternate proof of Proposition \ref{noteq1} to show that $S^i(K)\neq K$, for any $i\geq 1$ and any non-slice knot $K$ with a Legendrian diagram realizing $\tb(K)=2\g(K)-1$. Then, as in the proof of Proposition \ref{noteq2}, we appeal to Theorem 1, since $S$ is strong winding number one.

Together the results of this section constitute the main theorem. 
\section{Applications}\label{applications}

\begin{cor_1}We can choose $K$ to be topologically slice in the main theorem. This yields an infinite set of topologically slice knots $\{P^i(K)\}$ which are distinct in smooth (and exotic) concordance. \end{cor_1} 

\begin{proof}We can choose $K$ to be a topologically slice knot with a Legendrian realization such that $\tb(K)=2\g(K)-1$, such as the positive untwisted Whitehead double of any knot with this property \cite{Rud95}. If $K$ is topologically slice, then $P^i(K)$ is topologically concordant to $P^i(U)=\widetilde{P^i}$, which is unknotted. Thereore, for any such $K$, we generate an infinite set of smooth concordance classes of topologically slice knots. \end{proof}

\begin{cor_2}For any $L$--space knot $K$ with $\tb(K)>0$, there exist infinitely many prime knots with the same Alexander polynomial as $K$ which are not themselves $L$--space knots.\end{cor_2}

\begin{proof}For $L$--space knots $K$ which satisfy $\tb(K)=2\g(K)-1$ (such as the positive torus knots), this follows very easily from our main theorem. For the operator $P$, note that $\Delta_t(P^i(K))=\Delta_t(\widetilde{P^i})\Delta_t(K) = \Delta_t(K)$ since each $\widetilde{P^i}$ is unknotted \cite[Theorem II]{Seif50}. However, for an $L$--space knot $J$, $\tau(J)$ is equal to the degree of the symmetrized Alexander polynomial. Therefore, since each $P^i(K)$ has the same Alexander polynomial as $K$ but have distinct $\tau$--invariants, they are not $L$--space knots. 

In the more general case for an $L$--space knot with $\tb(K)>0$, we can still stabilize the Legendrian diagram to get a diagram with $\tb(\bar(K))=0$ and rotation number $R$. Using the same techniques as in the proof of the main theorem, we obtain that $$2i+\lvert R \rvert \leq 2\tau(P^i(K))-1$$

This shows that $\{\tau(P^i(K))\}_{i=0}^\infty$ is unbounded and monotone increasing, and that we can find a subsequence which is strictly increasing and bounded below by $\tau(K)$, which completes the proof.

If the patterns $P$ we use are unknotted as knots in $S^3$, the knots $P^i(K)$ are prime (see \cite[Theorem 4.4.1]{Crom04}). \end{proof}

\begin{cor_3}For any operator $P$ in the main theorem, the associated links $(P^i,\eta(P^i))$ yield distinct concordance classes of links with linking number one and unknotted components, which are each distinct from the class of the positive Hopf link. \end{cor_3}

\begin{proof}By Proposition 7.1 of \cite{CDR14}, since the functions $P^i$ (as well as $Q_i$, $R_i$ and their iterates) give non-trivial functions on $\mathcal{C}^\text{ex}$, the corresponding links cannot be smoothly concordant to the positive Hopf link. 
\end{proof}

\bibliographystyle{alpha}
\bibliography{knotbib}

\end{document}